\documentclass[12pt]{amsart}
\usepackage{hyperref}
\usepackage{amsthm,vmargin}
\usepackage{amsmath,amsfonts,xypic,mathrsfs}
\usepackage{epigraph}
\begin{document}
\newcommand{\Q}{{\mathbb Q}}
\newcommand{\C}{{\mathbb C}}
\newcommand{\R}{{\mathbb R}}
\newcommand{\N}{{\mathbb N}}
\newcommand{\Z}{{\mathbb Z}}
\newcommand{\F}{{\mathbb F}}
\newcommand{\bQ}{{\bar{\Q}}}
\newcommand{\brho}{{\bar{\rho}}}
\newcommand{\abs}[1]{\left|#1\right|}
\newcommand{\tr}{{\rm Tr}}
\newcommand{\image}{{\rm Image}}
\newcommand{\Gl}{{\rm GL}}
\newcommand{\Sl}{{\rm SL}}
\newcommand{\Gal}{{\rm Gal}}
\newcommand{\Li}{{\rm Li}}
\newcommand{\sA}{{\mathscr{A}}}
\newcommand{\sP}{{\mathscr{P}}}
\newcommand{\sS}{{\mathscr{S}}}
\newcommand{\sW}{{\mathscr{W}}}
\renewcommand{\wp}{{\mathfrak p}}
\renewcommand{\P}{{\mathbb P}}
\renewcommand{\O}{{\mathcal O}}
\newcommand{\Pic}{{\rm Pic\,}}
\newcommand{\Ext}{{\rm Ext}\,}
\newcommand{\rank}{{\rm rk}\,}
\newcommand{\sbull}{{\scriptstyle{\bullet}}}
\newcommand{\bX}{X_{\overline{k}}}
\newcommand{\ch}{\operatorname{CH}}
\newcommand{\tors}{\text{tors}}
\newcommand{\cris}{\text{cris}}
\newcommand{\alg}{\text{alg}}
\newcommand{\tX}{{\tilde{X}}}
\newcommand{\tL}{{\tilde{L}}}
\newcommand{\Hom}{{\rm Hom}}
\newcommand{\spec}{{\rm Spec}}
\let\isom=\simeq
\let\rk=\rank
\let\tensor=\otimes
\newcommand{\X}{\mathfrak{X}}
\newcommand{\mydot}{{\small{\bullet}}}
\newcommand{\congruent}{\equiv}
\renewcommand{\mod}{\mathop{{\rm mod}}}

\newtheorem{theorem}[equation]{Theorem}      
\newtheorem{lemma}[equation]{Lemma}          %
\newtheorem{corollary}[equation]{Corollary}  
\newtheorem{proposition}[equation]{Proposition}
\newtheorem{scholium}[equation]{Scholium}

\theoremstyle{definition}
\newtheorem{conj}[equation]{Conjecture}
\newtheorem*{example}{Example}
\newtheorem{question}[equation]{Question}

\theoremstyle{definition}
\newtheorem{remark}[equation]{Remark}

\numberwithin{equation}{subsection}

\title{Remarks on the Fourier coefficients of modular forms}
\author{Kirti Joshi}
\address{Math. department, University of Arizona, 617 N Santa Rita, Tucson
85721-0089, USA.} \email{kirti@math.arizona.edu}
\date{Version no. \jobname of October 3, 2009}

\begin{abstract}
We consider a variant of a question of N.~Koblitz. For an elliptic
curve $E/\Q$ which is not $\Q$-isogenous to an elliptic curve with
torsion, Koblitz has conjectured that there exists infinitely many
primes $p$ such that $N_p(E)=\#E(\F_p)=p+1-a_p(E)$ is also a prime.
We consider a variant of this question. For a newform $f$, without
CM, of weight $k\geq 4$, on $\Gamma_0(M)$ with trivial Nebentypus
$\chi_0$ and with integer Fourier coefficients, let
$N_p(f)=\chi_0(p)p^{k-1}+1-a_p(f)$ (here $a_p(f)$ is the
$p^{th}$-Fourier coefficient of $f$). We show under GRH and Artin's
Holomorphy Conjecture that there are infinitely many $p$ such that
$N_p(f)$ has at most $[5k+1+\sqrt{\log(k)}]$ distinct prime factors.
We give examples of about hundred forms to which our theorem
applies.
\end{abstract}

\maketitle \epigraph{To Pramodini J. Joshi,}{\emph{in memoriam
(1924-2009)}}


\section{Introduction}
\subsection{Koblitz' question for elliptic curves} For a natural number $n$,
let $\omega(n)$ denote the number of distinct prime factors of $n$
and let $\Omega(n)$ be the number of primes, dividing $n$ counted
with multiplicities.  Let $E/\Q$ be an elliptic curve and, for a
prime $p$ of good reduction of $E$. Let $N_p(E)=p+1-a_p(E)$ be the
number of points on the reduction of $E$ modulo $p$. Assume that $E$
is not $\Q$-isogenous to an elliptic curve with torsion.  In
\cite{koblitz88} the following question was studied: how often is
$N_p(E)$ a prime? In \cite{koblitz88} it was conjectured that this
happens infinitely often. Several authors (see \cite{miri01,murty04,
cojocaru05,steuding05,steuding05e}) have recently studied this
question.

\subsection{The example of the Delta function} Motivated by this question, we study a similar problem for
modular forms. Specifically, let $f(q)=\sum_{n=1}^\infty a_n(f)q^n$
be a cuspidal, normalized, new, Hecke eigenform on $\Gamma_0(M)$ of
weight $k\geq 2$ (here $M\geq 1$ is an integer) and Nebentypus
$\chi$, without complex multiplication and with integer Fourier
coefficients. Then we are interested in the number of prime factors
of $N_p(f)=\chi(p)p^{k-1}+1-a_p(f)$. In particular one can study
this for the Ramanujan modular form $\Delta(q)$ (see
\cite{ramanujan16}) of weight twelve
$$\Delta(q)=\sum_{n=1}^n\tau(n)q^n=q\prod_{n=1}^\infty\left(1-q^n\right)^{24},$$
on $\Sl_2(\Z)$. The example of Ramanujan form already shows that for
a given newform the numbers $N_p(f)=\chi(p)p^{k-1}+1-a_p(f)$ may be
composite for all primes $p$. Indeed Ramanujan's congruences
\cite{serre73} for $\tau(n)$ show that for all primes $p\geq 5$
\begin{equation}
N_p(\Delta)\congruent0\mod(2^5\cdot 3\cdot 691),
\end{equation}
so that we have $\omega(N_p(\Delta))\geq 3$ and
$\Omega(N_p(\Delta))\geq 7$ for all $p\geq 5$; and note that
$N_2(\Delta)=2^{11}+1-\tau(2)=3\cdot 691;
N_3(\Delta)=3^{11}+1-\tau(3)=2\cdot 691$. So we have in any case
that $\omega(N_p(\Delta))\geq 2$, thus the obvious variant of the
above question is trivially false for the Ramanujan modular form. In
\ref{refined-koblitz} we suggest a refined version of Koblitz's
conjecture which includes the behavior of the sort seen for
Ramanujan's Delta function.

\subsection{The main question}
However given a newform $f$ of weight $k$ on $\Gamma_0(M)$ and
Nebentypus $\chi$, with rational or integral Fourier coefficients,
we may still ask if the following weaker version of Koblitz'
question is still true: \emph{do there exist infinitely many primes
$p$ such that $N_p(f)=\chi(p)p^{k-1}+1-a_p(f)$ has bounded number of
prime factors?}. In this note we will prove that this is indeed the
case if we assume the Generalized Riemann Hypothesis and the
generalized Riemann hypothesis. We prove the following theorems:
\begin{theorem}\label{main}
Assume GRH and Artin's Holomorphy Conjecture for Artin $L$-functions
and suppose that $f$ is a newform satisfying hypothesis \ref{form}.
Let $$N_p(f)=\chi(p)p^{k-1}+1-a_p(f).$$  Then
\begin{enumerate}
\item there exists infinitely many primes $p$ such that
 $$\omega(N_p(f))\leq [5k+1+\sqrt{\log(k)}],$$
\item and there exists infinitely many primes $p$ such that
$$\Omega(N_p(f))\leq [8k+1+\sqrt{\log(k)}].$$
\end{enumerate}
\end{theorem}
\begin{corollary}\label{ramcor}
Assume GRH and Artin's Holomorphy Conjecture for Artin
$L$-functions. Let $\tau(n)$ be the Ramanujan  $\tau$-function.
Then there are infinitely many prime $p$ such that
$$
3\leq \omega(p^{11}-\tau(p)+1)\leq [61+\sqrt{\log(12)}]=[62.57\cdots]=62
$$ and an infinity of primes $p$ such that $$7\leq \Omega(N_p(\Delta))\leq [98.57\cdots]=98. $$
\end{corollary}
The proofs are along the lines of \cite{steuding05,steuding05e}. We
use a weighted sieve to arrive at this result. We prove more precise
versions of Theorems~\ref{main}, Theorem~\ref{ramcor} in
Theorem~\ref{main2} and Theorem~\ref{ramcor2} where we provide lower
bounds for the set of such primes. These results should be compared
with the known results for prime divisors of values of a fixed
irreducible polynomial (see \cite{halberstam-sieve}). In
Theorem~\ref{upperbound} we provide an upper bound which is of the
order of magnitude comparable to the lower bounds of
Theorems~\ref{main2} and Theorem~\ref{ramcor2}. In
section~\ref{examples} we give a table of about 100 modular forms of
various levels and weights where our theorem applies.

\subsection{Normal order of $N_p(f)$} In contrast to the above results, in Theorem~\ref{erdoskac}
we show that the average behavior of $\omega(N_p(f))$ is similar to
the average behavior of $\omega(p-1)$. More precisely we show (on
GRH) that $\omega(N_p(f))$ has normal order $\log\log(p)$. In fact a
version of Erdos-Kac Theorem holds for $\omega(N_p(f))$. This result
shows (again) that the set of primes $p$ for which $\omega(N_p(f))$
is bounded is of zero density (on GRH).

\subsection{Odds and ends}
We end with some remarks about the existence of forms $f$ for which
$\omega(N_p(f))\geq 2$ for all $p$. In fact we show (see
Remark~\ref{aremark}) that there exists a sequence of weights
$k_i\to \infty$ as $i\to\infty$ and a normalized cusp form ({\emph
but need not be eigenform}) $f_{k_i}$ of weight $k_i$ on $\Sl_2(\Z)$
with integer coefficients such that $\omega(N_p(f_{k_i}))\geq 2$ for
all $i$ and for all primes $p\geq 2$. So the numbers
$\omega(N_p(f))\geq 2$ for all $p\geq 2$ for infinitely many cusp
forms (normalized, but not necessarily eigenforms). In
Remark~\ref{bernoullibound} we record a bound for the number of
primes dividing the numerator of $B_n/n$.

\subsection{Acknowledgements}
We would like to thank M. Ram Murty for his comments. We have added
a refined version of Koblitz' conjecture in \ref{refined-koblitz} in
response to his question. We are grateful to the referee for many
suggestions and corrections. The section \ref{ongrh} on the results
assuming GRH (as opposed to GRH and Artin Holomorphy conjecture) was
added at the referee's suggestion.

\section{Statement of the main results}
\subsection{Hypothesis on our forms}\label{form} Let $k\geq 4$ be an integer. Let
\begin{equation}
    f(z)=\sum_{n=1}^\infty a_n(f)q^n
\end{equation} be a new, normalized, cuspidal,
Hecke eigenform of weight $k$ on $\Gamma_0(M)$ with Nebentypus
$\chi$. Following the usual convention we will simply call such a
form a \emph{newform on $\Gamma_0(M)$ with Nebentypus $\chi$}.
Throughout this paper we will assume that $f$ does not have complex
multiplication, CM for short, (see \cite{ribet76} for definition and
properties) and that $f$ has integer Fourier coefficients, i.e., we
will assume that $a_n(f)\in\Z$ for all $n\geq 1$. Since the field of
Hecke eigenvalues contains the field of values of $\chi$ (see
\cite{ribet76}), our assumptions restrict $\chi$ to be of order at
most two. Further by \cite[Remark 2, page 34]{ribet76} any form with
coefficient field $\Q$ and with Nebentypus of order two has CM.
\emph{So our assumption entails that $f$ is a newform of weight
$k\geq 4$ without complex multiplication on $\Gamma_0(M)$, with
trivial Nebentypus $\chi_0$ and with rational Fourier coefficients}.
We will work with such forms  throughout this paper.
 We note that if the level $M$ is square-free, then our assumption
that the form does not have complex multiplication is automatic (see
\cite{ribet76}). In section~\ref{examples} we give examples of about
hundred modular forms of square-free levels $\leq 21$ and weights
$\geq 4$ where our results apply. Our list is by no means
exhaustive. We note that we assume that $f$ has weight $k\geq 4$. A
form of weight two satisfying our hypothesis corresponds to elliptic
curves over $\Q$ (by the modularity theorem  (see \cite[Theorem
A]{breuil01} and references therein)) and has been covered by
\cite{miri01,murty04, cojocaru05,steuding05,steuding05e}.

A typical example of a form which satisfies the above hypothesis and
of particular interest to us is the Ramanujan cusp form of weight twelve given by
\begin{equation}\label{ramform}
  \Delta(z)=q\prod_{n=1}^\infty(1-q^n)^{24}=\sum_{n=1}^\infty\tau(n)q^n.
\end{equation}
The unique normalized cuspidal eigen forms $\Delta_k$, of weights $k$ for $k=12,16,18,20,22,26$ (with $\Delta_{12}=\Delta$) on $\Sl_2(\Z)$ are also forms which satisfy our hypothesis.

By the results of Deligne (\cite{deligne68}) we have for $f(z)$ satisfying our hypothesis
and for all prime $p\not|M$, that
\begin{equation}\label{deligne-ramanujan}
\abs{a_p(f)}\leq 2p^{(k-1)/2},
\end{equation}
and in particular for the Ramanujan modular form \eqref{ramform},
the famous assertion of Ramanujan:
\begin{equation}\label{ramest}
\abs{\tau(p)}\leq 2p^{(11)/2}.
\end{equation}

\subsection{The main result} For a form $f$ as in \ref{form}, we write
\begin{equation}\label{lfunction}
L(f,s)=\sum_{n=1}^\infty \frac{a_n(f)}{n^s}.
\end{equation}
Then by \cite{deligne68} we know that $L(f,s)$ converges for $Re(s)>\frac{k-1}{2}+1$  and has an
analytic continuation to a holomorphic function to all of $\C$. We now state the more precise
forms of the theorems stated in the introduction, which we will prove in the subsequent
subsections.

\begin{theorem}\label{main2}
Assume GRH and Artin's Holomorphy Conjecture for Artin $L$-functions
and suppose that $f$ is a newform satisfying hypothesis~\ref{form}.
\begin{enumerate}
\item  There exists a positive constant $c_1(f)$ depending on $f$ such that for $X\gg0$, one has
$$\#\left\{p\leq X: \omega(N_p(f))\leq [5k+1+\sqrt{\log(k)}]\right\}\geq c_1(f)\frac{X}{\log(X)^2}$$
\item There exists a positive constant $c_2(f)$ depending on $f$ such that for $X$ sufficiently large, one has
$$\#\left\{p\leq X: \Omega(N_p(f))\leq [8k+1+\sqrt{\log(k)}]\right\}\geq c_2(f)\frac{X}{\log(X)^2}$$
\end{enumerate}
\end{theorem}
\begin{corollary}\label{ramcor2}
Assume GRH and Artin's Holomorphy Conjecture for Artin
$L$-functions. Let $\Delta(q)=\sum_{n=1}^\infty\tau(n)q^n$ be the
Ramanujan cusp form on $\Sl_2(\Z)$ of weight $12$. Let $X$ be
sufficiently large. Then
\begin{enumerate}
 \item there exists a positive constant $c_1(\Delta)$ depending on $\Delta$ such
 that for $X\gg0$ one has
$$\#\left\{p\leq X: \omega(p^{11}+1-\tau(p))\leq 62\right\}\geq c_1(\Delta)\frac{X}{\log(X)^2}$$
\item there exists a positive constant $c_2(\Delta)$ depending on $\Delta$ such
that for $X\gg0$ one has
$$\#\left\{p\leq X: \Omega(p^{11}+1-\tau(p))\leq 98\right\}\geq c_2(\Delta)\frac{X}{\log(X)^2}$$
\end{enumerate}
\end{corollary}

Corollary~\ref{ramcor2} is, of course, immediate from Theorem~\ref{main2}.
\subsection{Upper bound}
We will also prove the following upper bound which shows that the
lower bounds of Theorem~\ref{main2} are of the right order of
magnitude, though a precise asymptotic formula seems out of reach at
the moment (even under GRH and Artin's Holomorphy Conjecture).

\begin{theorem}\label{upperbound}
Let $f(q)=\sum_{n=1}^\infty a_n(f)q^n$ be a newform satisfying the
hypothesis~\ref{form}. Assume GRH and Artin's Holomorphy Conjecture
for Artin $L$-functions. Then we have
$$\#\left\{ p\leq X: \Omega(N_p(f))\leq 9k-8\right\}\ll \frac{X}{(\log X)^2}.$$
\end{theorem}

\section{Nuts and bolts}
\subsection{}
Let $f(q)=\sum_{n=1}^\infty a_n(f)q^n$ be a newform satisfying our
hypothesis \ref{form}. Let us write
\begin{equation}
N_p(f)=\chi(p)p^{k-1}-a_p(f)+1
\end{equation}
Then by the theory of Hecke operators (for $p$ not dividing the
level) we know that $N_p(f)$ is the value at $Y=1$ of the
characteristic polynomial $Y^{2}-a_p(f)Y+\chi(p)p^{k-1}$ of the
Hecke operator $T_p$.

Write $Y^2-a_p(f)Y+\chi(p)p^{k-1}=(Y-\alpha_p(f))(Y-\beta_p(f))$.
Then we know by \cite{deligne68} that
$\abs{\alpha_p(f)}=\abs{\beta_p(f)}=p^{(k-1)/2}$

\subsection{} By the work of Deligne (and Deligne-Serre for weight one forms)
\cite{deligne68} we know that for every prime
$\ell$, associated to $f$ (as in \ref{form}) we have an $\ell$-adic Galois representation
\begin{equation}
\rho_{f,\ell}:\Gal(\bQ/\Q)\to \Gl_2(\Q_\ell)
\end{equation}
such that for every prime $p\not|M\ell$, we have
\begin{eqnarray}
\tr(\rho_{f,\ell}(Frob_p))&=&a_p(f)\\
\det(\rho_{f,\ell}(Frob_p))&=&\chi(p)p^{k-1}.
\end{eqnarray}

\subsection{}
Following Serre, Swinnerton-Dyer and Ribet (see
\cite{serre69,serre73,ribet76,ribet85}) we may also consider the
corresponding ``mod $\ell$'' representations. We recall the
following theorem from \cite{serre73,ribet85}.

\begin{theorem}
Let $f$ be a modular form as in \ref{form}. Let
$$
\brho_{f,\ell}:\Gal(\bQ/\Q)\to \Gl_2(\F_\ell)
$$
be the mod $\ell$ Galois representation associated to $f$. Then there exists an integer $m_f$
depending on $f$, such that for all $\ell$ not dividing $m_f$, we have
$$
\image(\brho_{f,\ell})=G_\ell,
$$
where
$$
G_\ell=\left\{g\in \Gl_2(\F_\ell): \det(g)\in (\F_\ell^*)^{(k-1)}\right\}.
$$
\end{theorem}

\subsection{}
Let $K_{f,\ell}=\bQ^{\ker(\brho_\ell)}$. Then $K_{f,\ell}$ is the
fixed field of $\ker(\brho_{f,\ell})$ and the extension $K_{f,\ell}$
is Galois with Galois group
\begin{equation}
\Gal(K_{f,\ell}/\Q)\simeq G_\ell.
\end{equation}

\subsection{}
More generally, let $\brho_{f,\ell^n}:\Gal(\bQ/\Q)\to
\Gl_2(\Z/\ell^n)$ be the mod $\ell^n$ Galois representation
associated to $f$. Let $K_{f,\ell^n}=\bQ^{\ker(\brho_{f,\ell^n})}$
be the fixed field of the kernel of $\rho_{\ell^n}$. Then
$K_{f,\ell^n}$ is Galois with Galois group contained in
$\Gl_2(\Z/\ell^n)$.

\begin{proposition}
Let $f$ be as in \ref{form}. Then the following are equivalent:
\begin{enumerate}
\item $\ell^n|N_p(f)$,
\item either $\brho_{f,\ell^n}(Frob_p)$ has an eigenvalue equal to
$1$, or $\brho_{f,\ell^t}(Frob_p)$ is unipotent for some $1\leq
t<n$.
\end{enumerate}
\end{proposition}
\begin{proof}
Since $\ell^n|N_p(f)$ if and only $\ell^n|(p^{k-1}-a_p(f)+1)$, so
$\ell^n|(1-\alpha_p(f))(1-\beta_p(f))$ where
$Y^2-a_p(f)Y+p^{k-1}=(Y-\alpha_p(f))(Y-\beta_p(f))$. Hence if
neither of $\alpha_p(f)-1,\beta_p(f)-1$ are divisible by $\ell^n$,
then $\alpha_p(f)\equiv 1\mod{\ell^r}$ and $\beta_p(f)\equiv
1\mod{\ell^s}$, for some $1\leq r,s<n$ and $r+s=n$. So that both
$\alpha_p(f)\equiv\beta_p(f)\equiv 1\mod{\ell^t}$ with
$t=\min(r,s)$. So this says that $\brho_{f,\ell^n}(Frob_p)$ acts as
a unipotent matrix modulo a suitable power, say $\ell^t$ with $1\leq
t<n$, of $\ell$. This proves the theorem.
\end{proof}
\begin{lemma}
Let $C_{\ell,n}\subset \Gl_2(\Z/\ell^n)$ be the subset of matrices
$g\in \Gl_2(\Z/\ell^n)$ such that $g$ either has an eigenvalue $1$
or for some $1\leq t<n$, the  image of $g$ under the natural map
$\Gl_2(\Z/\ell^n)\to\Gl_2(\Z/\ell^t)$, is identity. Then
$C_{\ell,n}$ is a conjugacy set (i.e. a union of conjugacy classes
of $\Gl_2(\Z/\ell^n)$.
\end{lemma}
\begin{proof}
This is clear.
\end{proof}
\subsection{}
We will use the following lemma.
\begin{lemma}\label{cardinaliylemma}
Let $\lambda=\gcd(k-1,\ell-1)$. Then
$$\#C_{\ell,1}=\frac{\ell^3-(\lambda+1)\ell}{\lambda}.$$
\end{lemma}
\begin{proof}
We count using method of \cite{washington86}. By definition
$\#C_{\ell,1}$ is the number of matrices in $G_\ell$ which have at
least one eigenvalue one. The number of such matrices is the number
of matrices which look like
$$\begin{pmatrix}
1 & *\\ 0 & u^\lambda\end{pmatrix}
$$ where $u\in\F_\ell^*$ times the number of one dimensional
subspaces of $\F_\ell\oplus\F_\ell$ minus the number of matrices
which get counted twice; the former number is $\ell
(\ell-1)/\lambda$ times $(\ell+1)$; while the number of matrices
which get counted twice are the  ones with both the eigenvalues
equal to one (and this number is $\ell$). So the number of matrices
in $C_{\ell,1}$ is $(\ell+1)\ell(\ell-1)/\lambda-\ell$. This easily
simplifies to $\frac{\ell^3-(\lambda+1)\ell}{\lambda}$.
\end{proof}
\begin{lemma}
We have
$$\#G_\ell=\frac{(\ell^2-1)(\ell^2-\ell)}{\lambda}.$$
\end{lemma}
\begin{proof}
This is clear from the exact sequence
$$1\to \Sl_2(\F_\ell)\to G_\ell\to (\F_\ell^*)^{\lambda}\to 1,$$
and the standard formula for computing the order of $\Sl_2(\F_\ell)$. This proves the assertion.
\end{proof}

\begin{lemma}\label{cardinalitylemma}
For $\ell \gg 0$, and $\lambda=\gcd(k-1,\ell-1)$ we have
$$\frac{\#C_{\ell,1}}{G_\ell}=\frac{1}{\ell}+O\left(\frac{1}{\ell^3}\right).$$
\end{lemma}
\begin{proof}
Clear from the fact that
$$\frac{\#C_{\ell,1}}{G_\ell}=\frac{\frac{1}{\ell}-\frac{\lambda+1}{\ell^3}}{1-\frac{1}{\ell}-\frac{1}{\ell^2}+\frac{1}{\ell^4}},$$
and so the assertion follows for $\ell\gg  0$ by expansion  of the
denominator.
\end{proof}

\section{The Weighted Sieve}
\subsection{} We will prove Theorem~\ref{main2} by using a suitable
weighted sieve due to Richert \cite{halberstam-sieve}. The sieve problem we encounter here is a
one dimensional sieve problem in the parlance of ``sieve methods'' and we will use notations
from \cite{halberstam-sieve} in this section.  We begin with the notations and conventions we
need to apply the results of \cite[Theorem~9.1, Lemma~9.1]{halberstam-sieve}.

\subsection{}\label{sieve1}
Let $\sA$ be a finite set of integers (need not be positive or distinct). Let $\sP$ be an
infinite set of prime numbers. For each prime $\ell\in\sP$, let $\sA_\ell=\left\{a\in\sA:
a\equiv 0\mod\ell\right\}$. We write
$$\#\sA=X+R_1$$
and
$$\#\sA_\ell=\delta(\ell)X+R_\ell$$
where $X$ is some approximation to $\sA$, and $\delta(\ell)X$ is some approximation to
$\sA_\ell$. For a square-free positive integer $d$ composed of primes of $\sP$, let
\begin{eqnarray}
\delta(d)&=&\prod_{\ell|d}\delta(\ell)\\
\sA_d&=&\cap_{\ell|d}\sA_\ell\\
R_d&=&\#\sA-\delta(d)X.
\end{eqnarray}
For $z>0$, let
$$P(z)=\prod_{\ell\in
\sP,\ell<z}\ell,$$
$$W(z)=\prod_{\ell\in\sP,\ell<z}\left(1-\delta(\ell)\right).$$
\subsection{Sieving hypotheses}\label{sieve2}
We will sssume that these satisfy the following hypothesis
\begin{description}
\item[($\Omega_1$)] there exists a constant $A_1\geq 0$ such that
$0\leq \delta(\ell)\leq 1-1/A_1$ for all $\ell\in\sP$.
\item[($\Omega_2(1,L)$)] there exists a constant $L\geq 1$ and $A_2$ such that
if $2\leq w\leq z$, then
$$-L\leq \sum_{2\leq w\leq z}\delta(\ell)\log\ell-\log(z/w)\leq
A_2,$$
\item[($R(1,\alpha)$)] there exists $0<\alpha<1$ and $A_3,A_4\geq 1$
such that, if $X\geq 2$ then
$$
\sum_{d<X^\alpha/(\log(X))^{A_3}}\mu(d)^23^{\omega(d)}\abs{R_d}\leq A_4 \frac{X}{(\log{X})^2}
$$
\end{description}

\subsection{The sifting function}
With $\sA, \sP$ as above, we consider a weighted sifting function of the following form
\begin{equation}
\sW(\sA,\sP,v,u,\lambda)=\sum_{a\in\sA,(a,P(X^{1/v})=1)}\left(1-\sum_{X^{1/v}\leq p
<X^{1/u},p|a,p\in\sP}  \beta(p,\lambda)\right),
\end{equation}
where
\begin{equation}
\beta(p,\lambda)=\begin{cases} \lambda\left(1-u\frac{\log p}{\log X}\right) &
\text{ if } X^{1/v}\leq p<X^{1/u}, p\in \sP\\
0 & \text{ otherwise}.
\end{cases}
\end{equation}
\subsection{} We recall the following form of Richert's weighted
one dimensional sieve from \cite[Theorem~9.1, Lemma~9.1]{halberstam-sieve}.
\begin{theorem}\label{richert}
Let the notations and conventions be as in \ref{sieve1},
\ref{sieve2}. Assume that hypothesis $\Omega_1$, $\Omega_2(1,L)$ and
$R(1,\alpha)$ hold for a set $\sA$ as in \ref{sieve1}, \ref{sieve2}.
Suppose further that there exists $u,v,\lambda\in \R$ and $A_5\geq
1$ such that
$$\frac{1}{\alpha}<u<v,\frac{2}{\alpha}\leq v\leq
\frac{4}{\alpha}, 0<\lambda<A_5.$$ Then
$$\sW(\sA,\sP,u,v,\lambda)\geq
XW(X^{1/v})\left(F(\alpha,v,u,\lambda)-\frac{cL}{(\log
X)^{1/14}}\right),$$ where
$$F(\alpha,v,u,\lambda)=\frac{2e^\gamma}{\alpha
v}\left(\log(\alpha v-1)-\lambda\alpha u\log\frac{v}{u}+\lambda(\alpha u-1)\log\frac{\alpha
v-1}{\alpha u-1}\right).$$ Here $\gamma$ is Euler's constant.
\end{theorem}
\subsection{Proofs of the main theorems}
We will now apply the Theorem~\ref{richert} to the following situation. We will take
$$\sA=\left\{N_p(f): p\leq X\right\},$$
and $$\sP=\left\{p: p\text{ a prime}\right\},$$ so that
$\#\sA=\pi(X)$ and  by the prime number theorem, we may write
$\#\sA=\#\sA_1=\frac{X}{\log(X)}+R_1$. By the Chebotarev density
theorem applied to the extensions $K_{\ell,f}$, we may take
$\delta(\ell)\frac{X}{\log(X)}$ as an approximation to $\#\sA_\ell$,
where $\delta(\ell)=\frac{\#C_{\ell,1}}{\#G_{\ell}}$ (see
\ref{cardinalitylemma}). To get uniform error term in the Chebotarev
density theorem for the extensions $K_\ell$ valid for a range of
$\ell$ we will need GRH and Artin's Holomorphy Conjecture
(especially the version of Chebotarev density theorem of
\cite{murty88}). To apply Theorem~\ref{richert} we have to verify
that the hypothesis $\Omega_1,\Omega_2(1,L)$, and $R(1,\alpha)$ hold
(see \ref{sieve2}). We will do this now.
\begin{lemma}
The hypothesis $(\Omega_1)$ holds with a suitable $A_1>0$, i.e., we
have
$$0\leq \delta(\ell)\leq 1-\frac{1}{A_1},$$
with a suitable $A_1$.
\end{lemma}
\begin{proof}
This is clear from the fact that
$\delta(\ell)=\frac{\#C_{\ell,1}}{\#G_\ell}=\frac{1}{\ell}+O(\frac{1}{\ell^3})$
(see Lemma~\ref{cardinalitylemma}).
\end{proof}
\begin{lemma}
The hypothesis $(\Omega_2(1,L))$ holds with a suitable $L$, i.e., there exists an $A_2\geq 1$
and an $L$ such that for $2\leq w\leq z$, we have
$$-L\leq \sum_{w\leq
p<z}\delta(\ell)\log(\ell)-\log\frac{z}{w}\leq A_2$$
\end{lemma}
\begin{proof}
This is again clear from Lemma~\ref{cardinalitylemma} and Mertens's
Theorem (see \cite[page 351]{hardy-wright}). Indeed
$\delta(\ell)=\frac{1}{\ell}+O(\frac{1}{\ell^3})$.
\end{proof}
The next step is to establish that $R(1,\alpha)$ holds. This is
where we use GRH and Artin's Holomorphy Conjecture. To prove
$R(1,\alpha)$ holds we need a form of Chebotarev density theorem
currently available under GRH and Artin's Holomorphy Conjecture (see
\cite{murty88}).
\begin{lemma}\label{main2lemma}
Let $f$ be a newform satisfying our hypothesis~\ref{form}. Let
$$R_d=\pi_f(X,d)-\delta(d)Li(X).$$ Assume GRH and Artin's Holomorphy Conjecture for all Artin
$L$-functions. Then hypothesis $(R(1,\alpha))$ holds with any
$\alpha<1/5$, i.e., we have for any $\alpha<1/5$:
$$\sum_{d<X^\alpha/(\log(X))^B}\mu(d)^23^{\omega(d)}\abs{R_d}\ll
\frac{X}{(\log{X})^2}$$
\end{lemma}
\begin{proof}
Observe that we have from \cite[page 260]{hardy-wright}, that
$$3^{\omega(n)}\leq d(n)^{3\log 3/\log 2} \ll n^\varepsilon.$$
Thus we have
$$\sum_{d<X^\alpha/(\log X)^B}\mu^2(d)3^{\omega(d)}\abs{R_d}\ll
\sum_{d<X^\alpha/(\log X)^B}d^\varepsilon\abs{R_d}.$$ Assuming GRH
and Artin's Holomorphy Conjecture and by \cite{murty88} we have
$$\abs{R_d}=O\left(d^{3/2}X^{1/2}\log(dX)\right).$$
So the sum in question is certainly
$$\ll\sum_{d<X^\alpha/(\log(X))^B}d^{3/2+\varepsilon}X^{1/2}\log(dX),$$
which is
$$\ll X^{1/2}\log(X)\sum_{d<X^\alpha/(\log
X)^B}d^{3/2+\varepsilon},$$ and this is, by partial summation,
$$\ll X^{1/2+(5/2)\alpha+\epsilon}.$$
For $\alpha<1/5$, the sum in the assertion is
$$\ll X^{1/4+\varepsilon}$$
and clearly this is certainly $\ll \frac{X}{(\log X)^2}$. This
proves the assertion.
\end{proof}
\begin{lemma}
For $X\gg 0$, we have
$$W(X)\gg \frac{1}{\log X}$$
\end{lemma}
\begin{proof}
This is clear from our estimates for $\delta(\ell)$ and Mertens's
Theorem~\cite[page 351]{hardy-wright}
\end{proof}

\subsection{Choice of sieve parameters}
Thus we can apply Theorem~\ref{richert} to our situation and we will make this explicit now. To
apply Theorem~\ref{richert} we need to choose $\alpha,u,v,\lambda$ satisfying conditions of the
theorem.

We choose as follows. We will take $k\geq 3$ and:
\begin{eqnarray}
\alpha&=&\frac{1}{5}-\frac{1}{5k}=\frac{k-1}{5k},\\
u&=&\frac{5k}{k-1}+\frac{1}{k-1}=\frac{5k+1}{k-1}\\
v&=&\frac{4}{\alpha}=\frac{20k}{k-1},\\
\lambda&=&\frac{1}{\sqrt{\log(k)}}.
\end{eqnarray}
Then we have
$$\frac{1}{\alpha}=\frac{5k}{k-1}<u=\frac{5k+1}{k-1}<v=\frac{4}{\alpha}=\frac{20k}{k-1}.$$
With these choices, we define
$$G_1(k)=F\left(\frac{k-1}{5k},\frac{20k}{k-1},\frac{5k+1}{k-1},\frac{1}{\sqrt{\log(k)}}\right)$$
explicitly this is given by
\begin{equation}
G_1(k)=\frac{e^\gamma\left(5k\log(3)\sqrt{\log(k)}+\log(15k)-(1+5k)\log(20k/1+5k)\right)}{10k\sqrt{\log(k)}}.
\end{equation}
It is clear that for $k\gg0$, one has $G_1(k)>0$ and numerically one
checks that for all $k\geq 3.039\cdots$, we have $G_1(k)>0$. Thus we
have that $F(\alpha,v,u,\lambda)>0$ for these choices of the
parameters. So we can apply Theorem~\ref{richert}, and note that by
the prime number theorem  (or by Chebyshev's Theorem) we have
$\#\sA\gg\frac{X}{\log X}$ and so we  deduce that
$$\sW(\sA,\sP,u,v,\lambda)\gg
\frac{X}{(\log X)^2}.$$

Now suppose that $p$ is such that $N_p$ has positive weight in the
sum $\sW(\sA,\sP,u,v,\lambda)$ then we claim that
$\omega(N_p(f))\leq [5k+\sqrt{\log(k)}]$. This will prove our
theorem. Indeed as the sum is positive for the above choices of
parameters, and so there are primes $p\leq X$ where $N_p(f)$ has
this property.

Now observe that for any such $p$, the``weight'' it contributes is
positive so
\begin{equation}\label{weightp}
0<{\bf w}(p)=1-\lambda\left(\sum_{X^{1/v}<q<X^{1/u},
q|N_p(f)}\left(1-u\frac{\log(q)}{\log(X)}\right)\right),
\end{equation}
and $N_p(f)$ has no prime divisors $q\leq X^{1/v}$; moreover any
prime divisor $q$ of $N_p(f)$ with $q\geq X^{1/u}$ we have
$$1-u\frac{\log(q)}{\log(X)}\leq 0,$$
and so even if we include the contribution of primes $q>X^{1/u}$, in
the sum \eqref{weightp}  we see that
$$0<1-\lambda\left(\sum_{q|N_p(f)}\left( 1 -u\frac{\log(q)}{\log(X)}\right)\right),$$
and this gives
\begin{eqnarray}
\omega(N_p(f))&=&\sum_{q|N_p(f)}1<u\sum_{q|N_p(f)}\frac{\log(q)}{\log(X)}+\frac{1}{\lambda}\\
&=&u\frac{\log(N_p(f))}{\log(X)}+\frac{1}{\lambda}\\
\end{eqnarray}
Now we use the Deligne-Ramanujan-Weil estimate:
$$N_p(f)=p^{k-1}+1-a_p(f)\leq p^{k-1}+1+2p^{(k-1)/2},$$
and we deduce that for sufficiently large $X$ we have
$$\sum_{q|N_p(f)}\frac{\log(q)}{\log(X)}\leq \frac{\log(N_p(f))}{\log(X)}$$
By the Deligne-Ramanujan-Weil estimate the last term is bounded by
$\frac{\log(X^{k-1}+1+2X^{(k-1)/2})}{\log(X)}$ and this is
$$\leq k-1+\frac{\log(1-X^{-(k-1)}+2X^{-(k-1)/2})}{\log(X)}$$
For any $\varepsilon>0$, we can find $X\geq X_0(k,\varepsilon)$ such that the second term in the
above is less than $\varepsilon\frac{(k-1)}{(5k+1)}$. So we have $$\omega(N_p(f))\leq
u\left((k-1)+\varepsilon\frac{k-1}{5k+1}\right)+\frac{1}{\lambda},$$ and now our assertion
follows using the fact that we have chosen $$u=\frac{(5k+1)}{k-1},
\lambda=\frac{1}{\sqrt{\log(k)}}.$$ Thus
$$\omega(N_p(f))\leq 5k+1+\sqrt{\log(k)}+\varepsilon.$$
For a fixed $k$ and $X$ suitably large, we may choose $\varepsilon>0$ so small that
$$[5k+1+\sqrt{\log(k)}+\varepsilon]=[5k+1+\sqrt{\log(k)}]$$ and so the first  assertion follows.

To prove the second assertion, we observe that we need to estimate the number of primes $p\leq
X$ which contribute to the sifting function with positive weights and have a prime divisor
$\ell|N_p(f)$ with $\ell^2|N_p(f)$ and $X^{1/v}\leq \ell\leq X^{1/u}$. We will follow the
argument of \cite{steuding05,steuding05e} to do this. We easily estimate the number of elements
of $C_{\ell,2}$ to as indicated in \cite{steuding05,steuding05e} and obtain
$$\frac{\#C_{\ell,2}}{\#G_\ell}=\frac{1}{\ell^2}+O(\ell^{-3}).$$
Thus we have
\begin{eqnarray*}
 \#\left\{ p\leq X: \ell^2|N_p(f), X^{1/v}\leq \ell\leq X^{1/u}\right\}&=&\sum_{X^{1/v}\leq \ell\leq X^{1/u}}\#\{p\leq X: \ell^2|N_p(f)\}.\\
&\ll&\frac{X}{\log(X)}\sum_{X^{1/v}\leq \ell\leq X^{1/u}}\frac{1}{\ell^2}+X^{1/2+\varepsilon}\sum_{X^{1/v}\leq \ell\leq X^{1/u}} \ell^3\\
&=&o\left(\frac{X}{\log(X)^2}\right),
\end{eqnarray*}
provided $u>8$. So we choose a new set of $u,v,\lambda$ as follows:
\begin{eqnarray}
 \alpha&=&\frac{k-1}{5k},\\
 u&=&\frac{8k+1}{k-1},\\
 v&=&\frac{16k}{k-1},\\
\lambda&=&\frac{1}{\sqrt{\log(k)}}.
\end{eqnarray}
Then we see that
$$G_2(k)=F\left(\frac{k-1}{5k},\frac{16k}{k-1},\frac{8k+1}{k-1},\frac{1}{\sqrt{\log(k)}}\right).$$
Explicitly we have
$$G_2(k)=\frac{e^\gamma}{16k\sqrt{\log(k)}}\left\{5\log(11/5)k\sqrt{\log(k)}+(3k+1)\log\left(\frac{11k}{3k+1}\right)-(8k+1)\log\left(\frac{16k}{8k+1}\right)
\right\}$$ and it is easy to see that $G_2(k)>0$ for $k\gg0$ and
numerically one has $G_2(k)$ is positive for $k\geq 4$.

Thus for these choices of $u,v,\lambda$ the  primes $p\leq X$ such that $N_p(f)$ has a small
square divisor do not contribute to the lower bound for the sifting function. Let $p$ be such a
prime, so that $N_p(f)$ makes a positive contribution to the sum $\sW$. Then  for such a $p$,
$N_p$ does not have any prime divisors less than $X^{1/v}$, and for primes $\ell|N_p(f)$ such
that $X^{1/v}<\ell<X^{1/u}$,  $\ell^2$ does not divide $N_p(f)$, while for the primes
$\ell>X^{1/u}$ which divide $N_p(f)$ (possibly dividing several times) the contribution to $\sW$
is always positive. So we may replace, in our previous argument, the function $\omega(N_p(f))$
by $\Omega(N_p(f))$. Indeed we have
$$0<{\bf w}(p)<1-\lambda\left(\sum_{q^m|N_p(f)}\left( 1 -u\frac{\log(q)}{\log(X)}\right)\right),$$
and this gives
\begin{eqnarray}
\Omega(N_p(f))&=&\sum_{q^m|N_p(f)}1<u\sum_{q^m|N_p(f)}\frac{\log(q)}{\log(X)}+\frac{1}{\lambda}\\
&=&u\frac{\log(N_p(f))}{\log(X)}+\frac{1}{\lambda}\\
\end{eqnarray}
Now we use the Deligne-Ramanujan-Weil estimate (note that
$\chi(p)=1$ for all but finite number of primes as $\chi$ is
trivial):
$$N_p(f)=\chi(p)p^{k-1}+1-a_p(f)\leq p^{k-1}+1+2p^{(k-1)/2},$$
and we deduce that for sufficiently large $X$ we have
$$\sum_{q|N_p(f)}\frac{\log(q)}{\log(X)}\leq \frac{\log(N_p(f))}{\log(X)}$$
By the Weil estimate the last term is bounded by $\frac{\log(X^{k-1}+1+2X^{(k-1)/2})}{\log(X)}$
and this is
$$\leq k-1+\frac{\log(1-X^{-(k-1)}+2X^{-(k-1)/2})}{\log(X)}$$
For any $\varepsilon>0$, we can find $X\geq X_0(k,\varepsilon)$ such
that the second term in the above is less than
$\varepsilon\frac{(k-1)}{(5k+1)}$. So we have $$\Omega(N_p(f))\leq
u\left((k-1)+\varepsilon\frac{k-1}{5k+1}\right)+\frac{1}{\lambda},$$
and now our assertion follows using the fact that we have chosen
$$u=\frac{(8k+1)}{k-1}, \lambda=\frac{1}{\sqrt{\log(k)}}.$$ Thus
$$\Omega(N_p(f))\leq 8k+1+\sqrt{\log(k)}+\varepsilon.$$
For a fixed $k$ and $X$ suitably large, we may choose $\varepsilon>0$ so small that
$$[8k+1+\sqrt{\log(k)}+\varepsilon]=[8k+1+\sqrt{\log(k)}]$$ and so the first  assertion follows.
Thus we see that
$$\#\left\{p\leq X: \Omega(N_p(f))\leq 8k+1+\sqrt{\log(k)}\right\}\gg\sW(\sA,\sP,v,u,\lambda)\gg c_2(f)\frac{X}{\log(X)^2}.$$
This proves our assertion.
\subsection{} Now we can prove the corollary. For the Ramanujan modular form
$$\Delta(q)=\sum_{n=1}^\infty\tau(n)q^n$$
we have $k=12$ and hence we deduce assuming  GRH and Artin's
Holomorphy Conjecture for Artin $L$-functions that, there exists
infinitely many prime $p$
$$3\leq \omega(p^{11}+1-\tau(p)) \leq [61+\sqrt{\log(12)}]=62.$$
So on GRH and Artin's Holomorphy Conjecture, there exists infinitely
many primes $p$, such that $p^{11}+1-\tau(p)$ has at most
$[61+\sqrt{\log(12)}]=62$ prime factors. Moreover there also exists
infinite many primes $p$ such that
$$7\leq \Omega(p^{11}+1-\tau(p)) \leq [97+\sqrt{\log(12)}]=98.$$
As indicated in the introduction, the bounds $\omega(N_p(\Delta))\geq 3$ and
$\Omega(N_p(\Delta))\geq 7$ are consequences of the Ramanujan congruences for $\Delta(q)$. This
completes the proof of Theorem~\ref{ramcor2}.

\section{Upper bounds}\label{upperbounds}
\subsection{}
We can also obtain an upper bound (again under GRH and Artin's
Holomorphy Conjecture) which shows that the lower bounds are of the
right order of magnitude. The upper bound is obtained using
Selberg's sieve.
\begin{theorem}\label{upperbound2}
Let $f(q)=\sum_{n=1}^\infty a_n(f)q^n$ be a newform satisfying
hypothesis~\ref{form}. Assume GRH and Artin's Holomorphy Conjecture
for Artin $L$-functions. Then we have
$$\#\left\{ p\leq X: \Omega(N_p(f))\leq 9k-8\right\}\ll \frac{X}{(\log X)^2}.$$
\end{theorem}
\begin{proof}
Let $\sA=\{N_p(f): p\leq X\}$, $\sP$ be the set of all primes. Let
$P(z)=\prod_{p<z,p\in\sP} p$; and let
$$\sS(\sA,\sP,z)=\sum_{a\in\sA, (a,P(z))=1}1.$$ Then \cite[Theorem~5.1,
Chapter~5]{halberstam-sieve} provides a convenient way for
estimating $\sS(\sA,\sP,z)$ under the hypothesis $(\Omega_1),
(\Omega_2(1,L))$ (see \ref{sieve2}). Since we have already verified
that these hypothesis hold in our setup, we can proceed to apply the
\cite[Theorem~5.1, Chaper~5]{halberstam-sieve} with $z=X^{1/9}$ and
we obtain
$$\sS(\sA,\sP,X^{1/9})\ll \frac{X}{(\log X)^2}.$$
Let  $p\leq X$ be a prime such that $N_p(f)$ contributes to the sum $\sS(\sA,\sP,X^{1/9})$, then
for any prime $\ell|N_p(f)$, we have $\ell>X^{1/9}$. So that we see that
$$N_p(f)>X^{\Omega(N_p(f))/9}.$$
On the other hand, for $p\leq X$, and $X\gg 0$, we have $N_p(X)\leq
2X^{k-1}$, so that we have
$$X^{\Omega(N_p(f))/9}<N_p(f)<2X^{k-1},$$
from which we see that one certainly has $\Omega(N_p(f))\leq 9(k-1)+1$. This proves the theorem.
\end{proof}

\section{Results on GRH}\label{ongrh}
\subsection{} We indicate briefly, the results one can obtain on GRH (as opposed to GRH and Artin Holomorphy
conjecture). As one might expect the results are weaker than the
ones obtained in the preceding sections. The main ingredient of the
proof which is influenced by GRH or GRH and Artin Holomorphy
conjecture is the Chebotarev density theorem of \cite{murty88}. On
GRH the error terms in the Chebotarev density theorem are weaker and
the bounds on $\omega(N_p(f))$ and $\Omega(N_p(f))$ are
correspondingly weaker. Both sets of hypothesis do however imply the
existence of infinitely many primes where $\Omega(N_p(f))$ (and
hence $\omega(N_p(f))$) is bounded by a constant depending on the
weight of $f$. To keep the discussion brief we will prove the result
for $\omega(N_p(f))$.
\begin{theorem}\label{main2grh}
Assume GRH  for Artin $L$-functions and suppose that $f$ is a
newform satisfying hypothesis~\ref{form}. Then there exists a
positive constant $c_1(f)$ depending on $f$ such that for $X\gg0$,
one has
$$\#\left\{p\leq X: \omega(N_p(f))\leq [8k+1+\sqrt{\log(k)}]\right\}\geq c_1(f)\frac{X}{\log(X)^2}.$$
\end{theorem}
\begin{corollary}\label{ramcor2grh}
Assume GRH  for Artin $L$-functions. Let
$\Delta(q)=\sum_{n=1}^\infty\tau(n)q^n$ be the Ramanujan cusp form
on $\Sl_2(\Z)$ of weight $12$. Let $X$ be sufficiently large. Then
there exists a positive constant $c_1(\Delta)$ depending on $\Delta$
such
 that for $X\gg0$ one has
$$\#\left\{p\leq X: \omega(p^{11}+1-\tau(p))\leq 98\right\}\geq c_1(\Delta)\frac{X}{\log(X)^2}.$$
\end{corollary}

\begin{proof}
The proofs of Theorem~\ref{main2grh} and Corollary~\ref{main2grh}
are similar to the proofs of Theorem~\ref{main2} and
Corollary~\ref{ramcor}. So we will indicate the changes required in
the argument and the choice of the sieving parameters we make which
allows us to arrive at the stated results. The change in our
hypothesis from Artin Holomorphy and GRH to GRH alone changes the
error term in the Chebotarev density theorem (see \cite{murty88}).
The change affects Lemma~\ref{main2lemma} which we replace by the
following Lemma~\ref{main2lemma-grh} given below.
\end{proof}
\begin{lemma}\label{main2lemma-grh}
Let $f$ be a newform satisfying our hypothesis~\ref{form}. Let
$$R_d=\pi_f(X,d)-\delta(d)Li(X).$$ Assume GRH  for all Artin
$L$-functions. Then hypothesis $(R(1,\alpha))$ holds with any
$\alpha<1/8$, i.e., we have for any $\alpha<1/8$:
$$\sum_{d<X^\alpha/(\log(X))^B}\mu(d)^23^{\omega(d)}\abs{R_d}\ll
\frac{X}{(\log{X})^2}$$
\end{lemma}
\begin{proof}
Observe that we have from \cite[page 260]{hardy-wright}, that
$$3^{\omega(n)}\leq d(n)^{3\log 3/\log 2} \ll n^\varepsilon.$$
Thus we have
$$\sum_{d<X^\alpha/(\log X)^B}\mu^2(d)3^{\omega(d)}\abs{R_d}\ll
\sum_{d<X^\alpha/(\log X)^B}d^\varepsilon\abs{R_d}.$$ Assuming GRH
by \cite{murty88} we have
$$\abs{R_d}=O\left(d^{3}X^{1/2}\log(dX)\right).$$
So the sum in question is certainly
$$\ll\sum_{d<X^\alpha/(\log(X))^B}d^{3+\varepsilon}X^{1/2}\log(dX),$$
which is
$$\ll X^{1/2}\log(X)\sum_{d<X^\alpha/(\log
X)^B}d^{3+\varepsilon},$$ and this is, by partial summation,
$$\ll X^{1/2+4\alpha+\epsilon}.$$
For $\alpha<1/8$, the sum in the assertion is
$$\ll X^{1/4+\varepsilon}$$
and clearly this is certainly $\ll \frac{X}{(\log X)^2}$. This
proves the assertion.
\end{proof}

Thus we can apply Theorem~\ref{richert} to our situation and we will make this explicit now. To
apply Theorem~\ref{richert} we need to choose $\alpha,u,v,\lambda$ satisfying conditions of the
theorem.

We choose as follows. We will take $k\geq 4$ and:
\begin{eqnarray}
\alpha&=&\frac{1}{8}-\frac{1}{8k}=\frac{k-1}{8k},\\
u&=&\frac{8k}{k-1}+\frac{1}{k-1}=\frac{8k+1}{k-1}\\
v&=&\frac{4}{\alpha}=\frac{32k}{k-1},\\
\lambda&=&\frac{1}{\sqrt{\log(k)}}.
\end{eqnarray}
Then we have
$$\frac{1}{\alpha}=\frac{8k}{k-1}<u=\frac{8k+1}{k-1}<v=\frac{4}{\alpha}=\frac{32k}{k-1}.$$
With these choices, we define
$$G_3(k)=F\left(\frac{k-1}{8k},\frac{32k}{k-1},\frac{8k+1}{k-1},\frac{1}{\sqrt{\log(k)}}\right)$$
explicitly this is given by
\begin{equation}
G_3(k)=\frac{e^{\gamma}}{32 k \sqrt{\text{log}(k)}} \left(8\text{log}(3) k  \sqrt{\text{log}(k)}+\text{log}(24 k)-(1+8 k) \text{log}\left(\frac{32 k}{1+8 k}\right)\right)
\end{equation}
and then it is clear that $G_3(k)>0$ for $k\gg0$ and numerically one
checks that for $k\geq 3.609\cdots$, we have $G_3(k)>0$. Thus we
have that $F(\alpha,v,u,\lambda)>0$ for these choices of the
parameters. So we can apply Theorem~\ref{richert}, and note that by
the prime number theorem  (or by Chebyshev's Theorem) we have
$\#\sA\gg\frac{X}{\log X}$ and so we  deduce that
$$\sW(\sA,\sP,u,v,\lambda)\gg
\frac{X}{(\log X)^2}.$$

Now the rest of the proof proceeds mutatis mutandis along the lines of Theorem~\ref{main2}.

\subsection{} Now we can prove the corollary. For the Ramanujan modular form
$$\Delta(q)=\sum_{n=1}^\infty\tau(n)q^n$$
we have $k=12$ and hence we deduce assuming  GRH that, there exists
infinitely many prime $p$ such that
$$3\leq \omega(p^{11}+1-\tau(p)) \leq [97+\sqrt{\log(12)}]=98.$$
As indicated in the introduction, the bound $\omega(N_p(\Delta))\geq
3$ is a consequence of the Ramanujan congruences for $\Delta(q)$.
This completes the proof of Theorem~\ref{ramcor2}.

\section{The normal order of $\omega(N_p(f))$}.
\subsection{} We show in this section that the average behavior of
$\omega(N_p(f))$ is similar to the average behavior of $\omega(n)$.
More precisely we have the following theorem.
\begin{theorem}\label{erdoskac} Let $f$ be a newform satisfying hypothesis of \ref{form}.
Assume GRH. Then we have
\begin{equation} \lim_{X\to\infty}\frac{1}{\pi(X)}\#\left\{p\leq X :
\frac{\omega(N_p(f))-\log\log(p)}{\sqrt{\log\log(p)}}\leq \gamma
\right\}=G(\gamma), \end{equation} where
$$G(\gamma)=\frac{1}{\sqrt{2\pi}}\int_{-\infty}^\gamma
e^{-\frac{t^2}{2}} dt.$$ \end{theorem}

\subsection{} This will be proved by the method of proof of
\cite[Theorem 1, page 156]{liu06}. The proof of
Theorem~\ref{erdoskac} is very similar to that of \cite[Theorem 1,
page 156]{liu06}. So we will only provide a brief sketch here. In
\cite[Theorem~3, page 160]{liu06} a general criterion is given for
an arithmetic function to satisfy the Erd\"os-Kac theorem and this
general result is applied to prove \cite[Theorem 1, page
156]{liu06}. We show here that these criteria are also valid (under
GRH) in the present case.  We show that Theorem~\ref{erdoskac} is a
consequence of the results of the previous sections and
\cite[Theorem~3, page 160]{liu06}. In this section we will use
notations from \cite{liu06}. To apply \cite[Theorem~3, page
160]{liu06} we have to verify the seven conditions (C) and (1)-(6)
of loc.cit. are satisfied. In the interest of brevity we will not
recall these here. But we will adhere to the notations of
\cite{liu06}. We take the set $S=\left\{p\leq X: \text{$p$ a
prime}\right\}$ and a function $N:S\to \N$ given by $p\mapsto
N_p(f)$. Thus condition (C) is obviously satisfied and we note that
by Lemma~\ref{cardinalitylemma}, the conditions (2), (4) and (5)
hold (by GRH and Mertens's Theorem). Thus we have to verify
condition (3) and (6). This needs the Chebotarev Theorem of
\cite{murty88}. In the notation of \cite{liu06} we have
$$e_\ell=\frac{X^{1/2}\ell^{3}\log(\ell^4X)}{\pi(X)},$$ and so
conditions (3) and (6) are verified exactly as in \cite{liu06} with
$\beta<\frac{1}{10}$. This completes our sketch of
Theorem~\ref{erdoskac}.

\section{Examples}\label{examples} We list examples of new forms of
low level where the theorems apply. We summarize the data in a
convenient form. Each row corresponds to a level. Each column
corresponds to a weight. The $(n,k)$-th entry lists the number of
distinct  newforms satisfying hypothesis of \ref{form}. The ninety
five forms in this list are all non-CM. The data provided here was
extracted from the Modular Forms Database \cite{mfd} and also using
Sage and PARI/GP software packages \cite{sage,PARI2}. We note that
the list is not exhaustive by any means. For instance, one can find
forms of weights $48$ on $\Gamma_0(2)$ and of weight $44$ on
$\Gamma_0(6)$ which satisfy hypothesis \ref{form} (see
\cite{tablesofforms}).

The empty spaces in the list are where our version of Sage failed to
provide conclusive answers or the database did not return values.
\begin{center}
\begin{tabular}{|c|c|c|c|c|c|c|c|c|c|c|c|c|c|c|}
\hline $N\downarrow,k\rightarrow$&4&6&8&10&12&14&16&18&20&22&24&26\\
\hline
1&0&0&0&0&1&0&1&1&1&1&0&1\\
\hline
2&0&0&1&1&0&2&1&1&2&2&0&1\\
\hline
3&0&1&1&2&1&&&&&&&\\
\hline
4&&&&&&&&&&&&\\
\hline
5&1&1&1&1&&&&&&&&\\
\hline
6&1&1&1&1&3&1&3&&&&&\\
\hline
7&1&1&1&&&&&&&&&\\
\hline 8&&&&&&&&&&&&\\ \hline 9&0&&&&&&&&&&&\\ \hline
10&1&3&1&3&3&&&&&&&\\ \hline 11&0&1&0&&&&&&&&&\\ \hline
12&&&&&&&&&&&&\\ \hline
13&1&0&1&&&&&&&&&\\
\hline
14&2&2&2&2&2&&&&&&&\\
\hline 15&2&2&2&2&1&&&&&&&\\ \hline
16&&&&&&&&&&&&\\ \hline 17&1&2&1&&&&&&&&&\\ \hline 18&&&&&&&&&&&&\\
\hline 19&1&2&&&&&&&&&&\\ \hline 20&&&&&&&&&&&&\\ \hline
21&2&4&1&1&2&&&&&&&\\
              \hline
            \end{tabular}
\end{center}

\section{Remarks and refinements}
\subsection{Refined version of Koblitz' conjecture}\label{refined-koblitz}
Let $f$ be a newform as in \ref{form}. Let $\ell$ be a prime and
suppose that for all but finite number of primes $p$ we have
$$\chi(p)p^{k-1}-a_p(f)+1\congruent0\mod{\ell},$$ then we say that
$\ell$ is an \emph{almost Eisenstein prime for $f$}. Suppose
$\nu(f)$ is the number of distinct almost Eisenstein primes for $f$.
The number of such primes is finite by the Theorem~\cite{ribet85}
(this is a subset of primes for which the mod $\ell$ Galois
representation associated to $f$ is reducible). Observe that
$\omega(N_p(f))\geq \nu(f)$ for all but finitely many $p$. Thus we
are led to the following reformulation of Koblitz's conjecture:
\begin{conj}\label{generalized-koblitz1} Let $f$ be a modular form as in \ref{form}. Let
$\nu(f)$ denote number of almost Eisenstein primes for $f$. Then for
$X\gg 0$ the number of primes $p\leq X$ such that
$\omega(N_p(f))=\nu(f)+1$ is at least $\gg\frac{X}{\log(X)^2}$.
\end{conj} For example, for $\Delta$, we have $\nu(\Delta)=3$ and so
the conjecture predicts that there are infinitely many primes $p$
where $N_p(\Delta)$ has exactly four distinct prime factors. Here
are all the primes $p\leq 16000$ with $\omega(N_p(\Delta))=4$.

\begin{center}
\begin{tabular}{|c|l|}
  \hline
  $p$ & $N_p(\Delta)$  \\
  \hline
  $5$ & $2^{10}\cdot3\cdot23\cdot691$  \\
  \hline
  $7$ & $2^9\cdot3^5\cdot23\cdot691$\\
  \hline
  $577$& $2^{13}\cdot3^7\cdot691\cdot190641378938814930857$\\
  \hline
  $1153$ & $2^{13}\cdot3^6\cdot691\cdot1160183970784175844330767$\\
  \hline
  $1297$ & $2^{11}\cdot3^6\cdot691\cdot16935741217449799251621239$\\
  \hline
  $3803$ &
  $2^8\cdot3\cdot691\cdot4534718285139898177401117938327717$\\
  \hline
  $5693$ &
  $2^{10}\cdot3\cdot691\cdot95907763393686429420185450510493683$\\
  \hline
  $11317$ &
  $2^{10}\cdot3^5\cdot691\cdot2268089547548261526855554962441076239$\\
  \hline
  $14437$ & $2^{10}\cdot3^6\cdot691\cdot11008825527208610156044088966777471773$\\
  \hline
  $15307$ &
  $2^8\cdot3\cdot5\cdot691\cdot251458672161512059369128893956312797721$\\
  \hline
\end{tabular}
\end{center}

We remark that Conjecture~\ref{generalized-koblitz1} includes the
following version of Koblitz's conjecture  for elliptic curves. The
original version of Koblitz' conjecture is made for elliptic curves
which are not $\Q$-isogenous to an elliptic curve with torsion.
\begin{conj}\label{generalized-koblitz2} Let $E/\Q$ be an elliptic curve. Let
$t_E=\#E(\Q)_{tor}$. Then for $X\gg 0$ the number of primes $p\leq
X$ such that $\omega(N_p(E))=\nu(t_E)+1$ is at least
$\gg\frac{X}{\log(X)^2}$.
\end{conj}

Indeed, if $q$ is a prime dividing the torsion subgroup of $E(\Q)$,
then it is well-known, that for primes $p$, of good reduction for
$E$, we have $N_p(E)=p+1-a_p(E)\congruent0\mod{q}$ (because torsion
injects into the $\F_p$-rational points). Hence $\nu(N_p(E))\geq
\omega(t_E)$ and any prime $q$ dividing $t_E$ is an almost
Eisenstein prime for $E$.

\subsection{Ramanujan-Serre forms}\label{ramanujan-serre-forms}
In the introduction we indicated that the  obvious variant Koblitz's
question is false for $\Delta(q)$. Here we show that there exists
infinitely many forms of weights $k_i$ tending to infinity (as
$i\to\infty$) such that for all primes $p$, we have
$\omega(N_p(f_{k_i}))\geq 2$. The construction depends on what we
should call \emph{weak Ramanujan-Serre forms}. A \emph{weak
Ramanujan-Serre} form $f$ for $\Sl_2(\Z)$ is a normalized cusp (but
not necessarily an eigen) form of weight $k$ on $\Sl_2(\Z)$ with
coefficients in $\Z[1/N]$ for some $N\geq 1$ and which is congruent
to the Eisenstein series of weight $k$ for every prime $q$ which
divides the numerator of $B_{k}/k$.  We say that a Hecke eigen form
of weight $k$ for $\Sl_2(\Z)$  is a \emph{strong Ramanujan-Serre}
form if it is congruent to some weak Ramanujan-Serre form for every
prime dividing $B_k/k$. The Ramanujan-Delta function is an example
of this; so are the cusp forms of weights less than 24. The
existence of such forms for low weights (the famous example being
$\Delta(q)$) is due to Ramanujan. The existence of weak
Ramanujan-Serre forms for all weights $k\geq 12$ is due to Serre
(see \cite{ribet76,khare00}), so the name weak Ramanujan-Serre forms
is appropriate; weak Ramanujan-Serre forms of weights $\leq 20$ are
also strong Ramanujan forms as there is only one prime dividing the
corresponding $B_k/k$. The unique normalized cusp form
$\Delta_{22}(q)$ of weight 22 is an example of a strong
Ramanujan-Serre form (there are two primes $131$ and $593$ dividing
$B_{22}/22$). Strong Ramannujan-Serre forms do not always exists.
For instance there are no  strong Ramanujan-Serre forms of weight
$k=24$. We do not know if strong Ramanujan-Serre forms always exist
(for all weights). We recall that a conjecture of Honda asserts that
there are no Hecke eigenforms for $\Sl_2(\Z)$ satisfying $k\geq 28$
and with integer Fourier coefficients.


\begin{remark}\label{aremark}
There exits an increasing, infinite sequence of integers $k_i\geq
12$  and weak Ramanujan-Serre forms  $f_{k_i}$ of weights
$k_i$ such that $\omega(N_p(f_{k_i}))\geq 2$. In particular
$N_p(f_{k_i})$ are not primes for all primes $p$.
\end{remark}
\begin{proof}
By a result of Serre \cite{ribet76,khare00}, we are assured of the
existence of weak Ramanujan-Serre forms. In other words there exists
cusp form $f_k(q)=\sum_{n=1}^\infty a_n(f_k)q^n$ of weight $k$ on
$\Sl_2(\Z)$ with coefficients in $\Z[1/N]$ for some $N\geq 1$, such
that for any prime $\ell$ dividing the numerator of $B_k/k$ we have
$$E_k=-B_k/k+\sum_{n=1}^\infty \sigma_{k-1}(n)q^n\congruent
f_k(q)\mod(\ell)$$ and thus we deduce that
$N_p(f_k)=p^{k-1}+1-a_p(f_k)$ has at least $\omega(B_k/k)$ prime
divisors for $p\geq 2$. Thus we see that
$$\omega(B_k/k)\leq \omega(p^{k-1}+1-a_p(f_k)).$$
So it will suffice to prove that there is a sequence of $k$ such that $\omega(B_k/k)\geq 2$.

To prove this we will use an old result of S.~Chowla (see \cite{chowla31}). It was shown in
\cite{chowla31} that if $p$ is an odd prime such that $p$ divides $\frac{B_n}{n}$ and such that
$p$ does not divide $2^n-1$, then $p$ divides the numerator of $B_{n+(p-1)i}$ for all $i\geq 0$.
We choose $n=24$ then $\frac{B_{24}}{24}=\frac{103\cdot2294797}{65520}$, while
$$2^{24}-1\not\congruent0\mod{103},$$ so by Chowla's Theorem, we see that for any $i\geq 0$,
$$\frac{B_{24+102i}}{24+102i}\congruent0\mod{103}.$$
Again for any $j\geq 0$, $691$ divides the numerator of
$\frac{B_{12+690j}}{12+690j}$. Since the arithmetic progressions
$\{12+690j\}_j$ and $\{24+102i\}_{i}$ have infinitely many common
elements (by the Chinese remainder theorem) so that we see that
there exists an infinite number of integers $i\geq 0$ such that
$$\frac{B_{12+690i}}{12+690i}\congruent0\mod{103\cdot691}.$$
Moreover, as $i\to\infty$ the numerator of $\frac{B_{12+690i}}{12+690i}$ goes to infinity as
well and so $\omega(B_{12+690i}/(12+690i))\geq 2$ as $i\to\infty$. Thus we may take
$k_i=12+690i$. Thus our assertion follows.
\end{proof}

\begin{remark}\label{bernoullibound}
The existence of weak Serre-Ramanujan forms $f_k$ as above has a curious consequence. In the
notation of the above proposition we have $\omega(B_k/k)\leq \omega(2^{k-1}+1-a_2(f_k))$. Now
for sufficiently large $n$, $\omega(n)\ll \frac{\log(n)}{\log(\log(n))}$. So we see that
$$\omega(B_k/k)\ll \frac{\log(2^{k-1}+1-a_2(f_k))}{\log\log(2^{k-1}+1-a_2(f_k))}\ll
\frac{k}{\log(k)}.$$ So we have for $k\gg 0$,
$$\omega(B_k/k)\ll \frac{k}{\log(k)}.$$
We do not know a better upper bound for the number of prime factors
of the numerator of $B_k/k$. We note that, as $\frac{B_k}{k}\asymp
k^k$, the normal order heuristics would suggests $\omega(B_k/k)\asymp
\log(k)$. But we do not know if $B_k/k$ satisfies the normal order estimates.
\end{remark}

%
%
\providecommand{\bysame}{\leavevmode\hbox to3em{\hrulefill}\thinspace}
\providecommand{\MR}{\relax\ifhmode\unskip\space\fi MR }
\providecommand{\MRhref}[2]{%
  \href{http://www.ams.org/mathscinet-getitem?mr=#1}{#2}
}
\providecommand{\href}[2]{#2}

\end{document}